\documentclass[11pt,hyp]{nyjm}

\usepackage{amsfonts}
\usepackage{amsmath}
\usepackage{amsthm}

\usepackage{fancyhdr}

\usepackage{hyperref}
\hypersetup{nesting=true,debug=true,naturalnames=true}
\usepackage{graphicx,amssymb,upref}

\let\<\langle
\let\>\rangle
\usepackage[all,pdf]{xy}
\UseComputerModernTips

\let\uml\"

\DeclareMathOperator{\bN}{\mathbb{N}}

\DeclareMathOperator{\bR}{\mathbb{R}}

\DeclareMathOperator{\cB}{\mathcal{B}}

\DeclareMathOperator{\cO}{\mathcal{O}}

\begin{document}

\title[The one-sided E.H.T. and Liouville numbers]{Everywhere divergence of the one-sided ergodic Hilbert transform and Liouville numbers}

\author{David Constantine}
\address{Department of Mathematics and Computer Science, Wesleyan University, 
265 Church St.,
Middletown, CT 06459}
\email{dconstantine@wesleyan.edu}

\author{Joanna Furno}
\address{Department of Mathematical Sciences, Indiana University-Purdue University Indianapolis, 
402 N. Blackford, LD 270
Indianapolis, IN 46202} 
\email{jfurno@iupui.edu}

\newcounter{main}
\theoremstyle{plain} \newtheorem{mainthm}[main]{Theorem}
\theoremstyle{plain} \newtheorem*{thm2}{Theorem 2}
\theoremstyle{plain} \newtheorem*{thm3}{Theorem 3}
\theoremstyle{plain} \newtheorem{thm}{Theorem}[section]
\theoremstyle{plain} \newtheorem{lemma}[thm]{Lemma}
\theoremstyle{plain} \newtheorem{prop}[thm]{Proposition}
\theoremstyle{plain} \newtheorem{cor}[thm]{Corollary}	
\theoremstyle{definition} \newtheorem{defn}[thm]{Definition}	
\theoremstyle{remark} \newtheorem{rmk}[thm]{Remark}
\theoremstyle{remark} \newtheorem{obs}[thm]{Observation}
\theoremstyle{remark} \newtheorem*{quest}{Question}
\theoremstyle{remark} \newtheorem{ex}[thm]{Example}

\begin{abstract}
We prove some results on the behavior of infinite sums of the form $\sum f\circ T^n(x)\frac{1}{n}$, where $T:S^1\to S^1$ is an irrational circle rotation and $f$ is a mean-zero function on $S^1$. In particular, we show that for a certain class of functions $f$, there are Liouville $\alpha$ for which this sum diverges everywhere. We also show that there are Liouville $\alpha$ for which the sum converges everywhere.
\end{abstract}

\maketitle
\tableofcontents


\section{Introduction}



\begin{quest}
Let $(X, \cB, \mu)$ be a probability measure space. Let $T$ be an invertible, measure-preserving, ergodic transformation on $(X, \cB, \mu)$. Let $\sum b_n$ be a positive, divergent series. Under what conditions do sums of the form

\begin{equation}\label{main_eqn}
	 \sum_{n=1}^\infty f\circ T^n (x) b_n  
\end{equation}
converge or diverge?
\end{quest}

In the specific case $b_n=\frac{1}{n}$, the sum \eqref{main_eqn} is known as the \emph{one-sided ergodic Hilbert transform} of $f$. Its convergence properties are of interest in part because convergence of equation \eqref{main_eqn} ensures convergence of the Birkhoff sums, and so results on the convergence or divergence of equation \eqref{main_eqn} provide a stronger version of Birkhoff's theorem, or indicate that no strengthening in this direction is possible.

For a fixed transformation $T$, conditions under which one-sided ergodic Hilbert transforms converge are very well studied. 
The first study of convergence was by Izumi in \cite{izumi}. 
Following Izumi's work, Halmos showed in \cite{halmos} that if the measure is non-atomic, then there are $L^2$ functions $f$ for which \eqref{main_eqn} diverges for almost every $x$. Similarly, Dowker and Erd{\"o}s showed in \cite{de} showed that if the measure is nonatomic, then there are $L^{\infty}$ functions $f$ for which \eqref{main_eqn} diverges for almost every $x$. 

Kakutani and Petersen in \cite{pk} extended the results of Dowker and Erd{\"o}s to show that mean zero functions for which the supremum of the norms of the partial sums of \eqref{main_eqn} is infinite for a.e. $x$ always exist in $L^\infty$ (see, similarly, \cite{krzyzewski}). In \cite{del_Junco_Rosenblatt}, del Junco and Rosenblatt show that a.e.-divergence is generic in the function $f$. These results have been of the following form: given $T$, $b_n$, $x$, there exist mean zero functions $f$ (with some regularity properties specified -- continuous $f$ are possible, see \cite{del_Junco_Rosenblatt}) such that the sum \eqref{main_eqn} diverges. 
Often, a careful study of the regularity of $f$ or its Fourier coefficients plays a key role.

 Divergent sum behavior has subsequently been investigated in very general contexts. \cite{Berkes_Weber} provides a monograph-length treatment of the subject, and \cite{al} provides a good overview of work in the general setting of contracting operators on Banach spaces. Further results can also be found in \cite{cuny}, \cite{cl}, \cite{ccl}, and \cite{lss}.

In contrast to these general settings and the non-constructive proofs that appear, we will demonstrate divergence in the case of the simple and well-understood dynamics of circle rotations, with specific random variables $f$ that are quite simple -- essentially indicator functions of intervals. 

Let $\alpha \in (0,1)$ be an irrational number and let $T:=R_\alpha$ be the rotation by $\alpha$ on $S^1 = \mathbb{R}/\mathbb{Z}$.  That is, $Tx = x+\alpha \mbox{ (mod 1)}$.  Let $f:S^1\to \mathbb{R}$ be any mean zero function on $S^1$ with respect to the Lebesgue measure.

\begin{defn}
Given $T=R_{\alpha}$, we say $\alpha$ is \emph{convergent} or \emph{divergent} for a point $x$ and a function $f$ according to the behavior of the series \eqref{main_eqn}.
\end{defn}

In Section \ref{sec:divergent} below we prove our first main theorem:

\begin{mainthm}\label{divergent theorem}
Let $b_n=\frac{1}{n}$ and let $f=2\chi_U -1$ where $U$ is a finite union of disjoint intervals with $m(U)=1/2$. Then there are irrational $\alpha$ which are divergent for all points $x$. Such an $\alpha$ can be provided explicitly in terms of its continued fraction expansion.
\end{mainthm}

We will call functions like $f$ \emph{mean-zero indicator functions} for a finite union of intervals. Given $\alpha$, the reader can easily construct a mean-zero indicator function for a \emph{countable} union of intervals for which divergence everywhere will fail.

Previous work on this problem has mainly used tools from functional analysis, and has produced results for almost every $x$. Some exceptions to this can be found in \cite{fs} and \cite{Berkes_Weber}, but in each case some regularity or $f$ or additional assumptions on its Fourier series are required. A key difference in Theorem \ref{divergent theorem} -- and our subsequent theorems -- is that we prove divergence for all $x$.

There are several straightforward consequences of the proof of Theorem \ref{divergent theorem}. First, the set of divergent $\alpha$ is dense. Second, $\alpha$ can be taken to depend on $f$ only through the number of intervals in $U$. Third, one can replace $\frac{1}{n}$ with any sequence $b_n$ such that $\sum n(b_{n+1}-b_n)$ diverges.

In Section \ref{Liouvillesec}, we investigate the set of divergent $\alpha$, for functions of the type described above. As noted previously, this is a Lebesgue measure zero set; we obtain the following stronger result, for $b_n=\frac{1}{n}$:

\begin{mainthm}\label{thm2}
Any divergent $\alpha$ is a Liouville number. Hence the set of divergent $\alpha$ has Hausdorff dimension zero.
\end{mainthm}

This theorem was previously known. In \cite{pk}, Kakutani and Petersen note that convergence holds for non-Liouville $\alpha$ when $f$ is the (mean-zero) indicator function of an interval. They remark that this result follows from number-theoretic results on the discrepancy for a non-Liouville number. As they do not detail the proof and as we have been unable to find it elsewhere in the literature, we include it in Appendix \ref{appendix}. This proof does give convergence for all $x$.

Our proof of Theorem \ref{thm2} is longer, but it does have the advantage that the proof provides a way to glean some information about the sum behavior of the Liouville numbers. For instance, we can prove the following theorem in Section \ref{conv Liou}, which does not seem to follow from the type of arguments outlined in Appendix \ref{appendix}:


\begin{mainthm}
There are convergent Liouville numbers.
\end{mainthm}

We show specifically how to produce such numbers using the mechanisms developed in our proof of Theorem \ref{thm2}.

\subsection{Acknowledgements}

We would like to thank Adam Fieldsteel and Randy Linder for bringing this problem to our attention. We would also like to thank Jon Chaika, Felipe Ram\'irez, David Ralston, Vitali Bergelson, and Joseph Rosenblatt for helpful conversations and comments.


\section{Setup}

We fix the following notation throughout the paper:

\begin{itemize}
	\item  $\alpha \in (0,1)$ is an irrational number.
	
	\item We write $\alpha = [a_1 a_2 a_3 \ldots  ]$ for the continued fraction expansion of $\alpha$.  

	\item $S[a_1\ldots a_n]$ is the set of all irrational $\alpha$ with a continued fraction expansion beginning with $[a_1\ldots a_n]$.
	
	\item $\frac{p_n}{q_n} = [a_1 \ldots a_n]$ is the $n^{th}$ convergent to $\alpha$. The $q_n$ can be determined from the $a_n$ via the recurrence relation
\[ q_n=a_nq_{n-1}+q_{n-2}, \quad q_1:=a_1,\quad  q_0:=1. \]

	\item $T=R_\alpha$. 
	
	\item $\langle \langle x \rangle\rangle$ denotes the distance from $x$ to $0$ in $S^1$.


	\item $U=\cup_{l=1}^B I_l$ is a finite union of intervals in $S^1$.

	\item For a fixed $x \in S^1$ and for integers $j_1 \leq j_2$, let 
	\[\cO_{\alpha}[j_1, j_2] = \left\{R_{\alpha}^i(x) : j_1 \leq i \leq j_2 \right\}.\] 
	We call $\cO_{\alpha}[j_1, j_2]$ an \emph{orbit segment} with \emph{length} $j_2 - j_1 + 1$.

	\item 
	If $\cO_{\alpha}[j_1, j_2]$ is an orbit segment, let 
	\[s(\cO_{\alpha}[j_1, j_2]) = \sum_{i=j_1}^{j_2} f\circ T^i(x).\]

\end{itemize}

The following basic facts relate the continued fraction expansion of $\alpha$ to the dynamics of $T$:

\begin{itemize}
	\item For $n$ odd, $T^{q_n}(0) = q_n \alpha$ is closer to 1 than to 0, and for $n$ even, $T^{q_n}(0)$ is closer to 0 than to 1.  In other words, the $n^{th}$ convergent to $\alpha$ is an overestimate for $n$ odd and an underestimate for $n$ even (see, e.g. \cite[Thm 8]{khinchin}).
	\item For irrational $\alpha$, $\langle\langle T^{q_n}0\rangle\rangle < \langle\langle T^{m}0\rangle\rangle$ for any $m<q_{n+1}$. In other words, the convergents are precisely the best approximations of the second kind to $\alpha$, i.e. $0<m\leq q_n$ and $\frac{l}{m}\neq\frac{p_n}{q_n}$ imply $|m\alpha-l|>|q_n\alpha-p_n|$ (see, e.g. \cite[Thms 16 \& 17]{khinchin}).
\end{itemize}


\section{Divergent $\alpha$ exist}\label{sec:divergent}

\subsection{The basic idea}

The proof of Theorem \ref{divergent theorem} is driven by a simple idea. The convergent $\frac{p_n}{q_n}$ is the best rational approximation of the second kind for $\alpha$ with denominator less than $q_{n+1}=a_{n+1}q_n+q_{n-1}$.  If $a_{n+1}$ is quite large, then this approximation must be quite good (i.e. $\langle\langle q_n\alpha \rangle\rangle$ is quite small) and the orbits of $x$ under $T$ and $R_{p_n/q_n}$ track closely for a long time. The orbit of $x$ under $R_{p_n/q_n}$ is periodic, hitting $q_n$ points. Suppose that $q_n$ is odd. Then for parity reasons, values of $f$ over the orbit of $x$ under the rational rotation include at least one more $+1$ than $-1$, or at least one more $-1$ than $+1$, over each $q_n$ steps. This constant rate of accumulation of extra $+1$'s or $-1$'s causes the sum for the rational rotation to diverge. With $a_{n+1}$ quite large, orbit of the irrational rotation tracks that of the rational rotation closely for a long time, and we will show that it must also accumulate extra $+1$'s or $-1$'s at a constant rate for a long stretch of orbit. This will drive divergence of the sum.


\subsection{Lemmas}

The proof of our first lemma is clear.

\begin{lemma}\label{notzero}
If $\cO_{\alpha}[j_1, j_2]$ is an orbit segment of odd length, then $s(\cO_{\alpha}[j_1, j_2])\neq 0.$
\end{lemma}

Consider the orbit segment $\cO_{\alpha}[1,q_{n+1}]$. We decompose it into segments $\sigma_0 = \cO_{\alpha}[1,q_{n-1}]$ and $\sigma_l=\cO_{\alpha}[q_{n-1}+(l-1)q_n+1, q_{n-1}+lq_n]$ for $l=1,2,\ldots, a_{n+1}$. Note that the parity of the orbit segment $\sigma_l$ depends on $\alpha$ only through the value of $q_n$, i.e. only through the values of the first $n$ terms in the continued fraction expansion of $\alpha$.

\begin{lemma}\label{changes}
Let $C=\{ l\in [1,a_{n+1}-1]: s(\sigma_l) \neq s(\sigma_{l+1})\}$. i.e. $C$ is the set of $l$ at which $s(\sigma_l)$ changes. Then $|C|\leq 2B$, where $B$ is the number of intervals forming $U$.
\end{lemma}

\begin{proof}

First, $\sigma_{l+1} = R_{q_n\alpha}\sigma_l = R_{\pm\langle\langle q_n\alpha\rangle\rangle} \sigma_l$ with sign depending on whether $\frac{p_n}{q_n}$ over- or underestimates $\alpha$. As $\sigma_1$ is an orbit segment of length $q_n$, the minimum distance between its points is $\langle\langle q_{n-1}\alpha \rangle\rangle > a_{n+1}\langle\langle q_n\alpha \rangle\rangle$. Thus, if $z$ is an endpoint of some interval in $U=\cup_{l=1}^B I_l$, over the $a_{n+1}$ rotations of $\sigma_1$ by $R_{\pm\langle\langle q_n\alpha\rangle\rangle}$, each of which moves points by $\langle \langle q_n \alpha \rangle\rangle$, at most one orbit point can cross $z$. Moreover, $s(\sigma_l)\neq s(\sigma_{l+1})$ only if an orbit point crosses the endpoint of an interval under the rotation $R_{\pm\langle\langle q_n\alpha\rangle\rangle}$ of $\sigma_l$. As there are $2B$ endpoints, and each is crossed at most once, $|C|\leq 2B$.
\end{proof}

We also need the following technical lemma and its corollary.

\begin{lemma}\label{partslemma}
Suppose that $c_m$ is a sequence of $\pm1$'s. Let $s_n=\sum_{m=1}^n c_m$ and suppose that $\frac{s_n}{n}\to L \neq 0$. Then for any $\kappa>0$, $\sum c_m\frac{1}{m+\kappa}$ diverges.
\end{lemma}

\begin{proof}
By the summation by parts formula, 

\begin{align}
	\sum_{m=1}^N c_m\frac{1}{m+\kappa}= & \frac{s_N}{N+\kappa+1}-\sum_{m=1}^Ns_m(\frac{1}{m+\kappa+1}-\frac{1}{m+\kappa}) \nonumber \\
		=& \frac{s_N}{N+\kappa+1}+\sum_{m=1}^N \frac{s_m}{m+\kappa}\frac{1}{m+\kappa+1}. 
		\label{partssum}
\end{align}
Pick $N_1$ so large that for all $m>N_1$, $\frac{s_m}{m+\kappa}\geq\frac{L}{3}$ if $L>0$, or $\frac{s_m}{m+\kappa}\leq\frac{L}{3}$ if $L<0$. Then for all $m>N_1$, $\frac{s_m}{m+\kappa}\frac{1}{m+\kappa+1}$ is at least $\frac{L}{3}\frac{1}{m+\kappa+1}$ if $L>0$, and at most $\frac{L}{3}\frac{1}{m+\kappa+1}$ if $L<0$. In either case, the comparison sums diverge, so by making $N$ sufficiently large, the term $\sum_{m=1}^N\frac{s_m}{m+\kappa}\frac{1}{m+\kappa+1}$ can be made arbitrarily large in absolute value. Since the term $\frac{s_N}{N+\kappa+1}$ approaches $L$, the expression in (\ref{partssum}) grows arbitrarily large in absolute value, showing divergence.
\end{proof}

The following is an immediate corollary of the above proof.

\begin{cor}\label{sumtool}
Fix $L \in \bR$ and $\kappa, N_1 \in \bN$. 
Suppose that $c_m$ is a sequence of $\pm1$'s such that the partial sums $s_n=\sum_{m=1}^n c_m$
satisfy $\frac{s_n}{n+\kappa}\geq \frac{L}{3}$ for all $n>N_1$ if $L>0$, or satisfy 
$\frac{s_n}{n+\kappa}\leq \frac{L}{3}$ for all $n> N_1$ if $L<0$.
Then for any $A > 0$ there is an $N^{*}$
such that $|\sum_{m=1}^{N^*}c_m\frac{1}{m+\kappa}|>A$, where $N^{*}$ depends only on $L$, $\kappa$, $N_1$, and $A$.
\end{cor}

The crucial fact here is that $N^*$ can be determined by the values of $L$, $\kappa$, $N_1$, and $A$,
without dependence on the specific terms of the sequence $c_m$. This will play a key role in the proof of Theorem \ref{divergent theorem}.

\begin{proof}
We continue using the summation by parts formula in equation (\ref{partssum}). If $n > N_1$, then the terms $\frac{n}{n + \kappa + 1}$ and $\sum_{m=N_1 + 1}^n \frac{s_m}{m+\kappa}\frac{1}{m+\kappa+1}$ have the same sign as $L$, by our hypothesis on the 
sequence. On the other hand, the middle terms $\sum_{m=1}^{N_1} \frac{s_m}{m+\kappa}\frac{1}{m+\kappa+1}$ might have a different sign. Since $|s_m| \leq m$ for any
sequence $c_m$, we can bound this potential cancellation without dependence on the sequence. By our hypothesis on the sequence $c_m$, we can choose $n > N_1$ such that the magnitude of $\sum_{m=N_1 + 1}^n \frac{s_m}{m+\kappa}\frac{1}{m+\kappa+1}$ becomes arbitrarily large, without further assumptions on the sequence.

To be more concrete in our choices, let $E = \sum_{m=1}^{N_1} \frac{m}{m+\kappa}\frac{1}{m+\kappa+1}$. Since $|s_m| \leq m$, we have the bound $|\sum_{m=1}^{N_1} \frac{s_m}{m+\kappa}\frac{1}{m+\kappa+1}| \leq E$. 
Choose $N_2 \in \bN$ such that $\sum_{m=1}^{N_2}\frac{1}{m+\kappa+1} > (A+E)|3/L|$. Take $N^{*} \geq \max\left\{N_1, N_2 \right\}$. Since $N^{*}\geq N_1$, the term $\frac{N^*}{N^* + \kappa + 1}$ has the same sign and is larger in magnitude than $L/3$. Since $N^{*} \geq N_2$, we have 
\begin{eqnarray*}
\left|\sum_{m=1}^{N^*} \frac{s_m}{m+\kappa}\frac{1}{m+\kappa+1}\right| &\geq &
\left|\sum_{m=1}^{N^*} \frac{L}{3}\frac{1}{m+\kappa+1}\right|\\
 &\geq & A + E.
\end{eqnarray*}
Additionally, these two terms on the right side of equation (\ref{partssum}) have the same sign.
Hence, $|\sum_{m=1}^{N^*}c_m\frac{1}{m+\kappa}| > |L/3| + A > A$.
\end{proof}


\subsection{Proof of Theorem \ref{divergent theorem}}

Suppose $q_n$ is odd. By Lemma \ref{changes}, if $a_{n+1}$ is much larger than $2B$, then there will be a long stretch of indices $l$ during which $s(\sigma_l)$ does not change. Over this stretch, the sum accumulates extra $+1$'s or $-1$'s at a rate of at least $1/q_n$. This, via Corollary \ref{sumtool} will drive the divergence of the sum.

We will find $\alpha$ satisfying the requirements of Theorem \ref{divergent theorem} by inductively defining its continued fraction expansion. This requires some care in the sort of arguments we can make. Recall that $S[a_1\ldots a_n]$ is the set of all irrational $\alpha$ with a continued fraction expansion beginning with $[a_1\ldots a_n]$. To define $a_{n+1}$ in our inductive scheme, we must make arguments which rely only on fixed data (such as the value of $B$) and on statements which are true for all $\alpha\in S[a_1\ldots a_n]$. For example, since $\alpha$ is not yet known, we do not know the exact sign pattern $(f\circ T^nx)_n$. We must instead rely on information about it gleaned only from the first $n$ terms of the continued fraction expansion, such as Lemma \ref{notzero} (applied to $\sigma_l$) and Lemma \ref{changes}.

\begin{proof}[Proof of Theorem \ref{divergent theorem}.]
We define $\alpha$ by producing its continued fraction expansion inductively. We prove divergence via the Cauchy criterion, showing that for any $q_n$, there are $n_2>n_1>q_{n}$ such that $|\sum_{n=n_1}^{n_2} f\circ T^n(x)\frac{1}{n}|>1$.

Let $a_1=1$. Then $q_1 = 1$. By choosing all $a_n$ even for $n\geq 2$, we can ensure that all $q_n$ are odd, simplifying the proof below. Since the argument below proceeds by choosing $a_{n+1}$ sufficiently large, this causes no problems.

For some $n \geq 1$, suppose that we have chosen $a_1, \ldots, a_n$ with $a_1=1, a_n$ even for $n\geq 2$. Then $q_n$ is odd. Moreover, note that all $\alpha \in S[a_1 \dots a_n]$
have the same $q_n$ and $q_{n-1}$.

Let $\kappa = q_{n-1}$.
For any $\alpha \in S[a_1\ldots a_n]$, let $c_m'(\alpha) = f\circ T^{\kappa+m}(x)$. The sequence $c_m'(\alpha)$ depends, of course, on the value of $\alpha$, but for all $\alpha\in S[a_1\ldots a_n]$, it accumulates at least one extra $+1$ or $-1$ over every orbit segment $\sigma_l$, by Lemma \ref{notzero}. 
Now let

\begin{equation*}
	c_m(\alpha)=\left\{
	\begin{array}{rl}
		c_m'(\alpha) & \mbox{ if } m\in \sigma_l \mbox{ with } s(\sigma_l)>0 \\
		-c_m'(\alpha) & \mbox{ if } m\in \sigma_l \mbox{ with } s(\sigma_l)<0.
	\end{array}\right.
\end{equation*}
That is, we switch the signs of $c_m'(\alpha)$ on an orbit segment $\sigma_l$ precisely when necessary to ensure the sum of $c_m(\alpha)$ over the segment is positive. Now, for all $\alpha \in S[a_1\ldots a_n]$, $c_m(\alpha)$ accrues at least one extra $+1$ over each segment $\sigma_l$. 

Take $L = 1/q_n$ and $N_1 = q_n^2$. As noted, $\kappa = q_{n-1}$. 
Let $s_m = \sum_{i=1}^{m}c_m$. First, we check that the hypotheses of Corollary \ref{sumtool} hold at multiples of $q_n$ greater than $N_1$. Then, we check that the hypotheses of Corollary \ref{sumtool} hold between multiples of $q_n$. First, at each $m=lq_n$, $\frac{s_m}{m+\kappa}\geq \frac{l}{lq_n+q_{n-1}}$. If $m>N_1=q_n^2$, $\frac{s_m}{m+\kappa} >\frac{L}{3}$ since $q_n^2>q_{n-1}+1$.

Now suppose that $lq_n < m < (l+1)q_n$ with $l\geq q_n$. Then $s_m>l-\lfloor \frac{q_n}{2}\rfloor$ since the sum has accumulated at least $l$ extra $+1$'s and there are at most $\lfloor \frac{q_n}{2}\rfloor$ of the opposite signs between $lq_n$ and $(l+1)q_n$ which can cancel them out. From this we compute:

\begin{align}
	\frac{s_m}{m+\kappa}  >& \frac{l}{m+q_{n-1}} - \frac{\lfloor q_n/2 \rfloor}{m+q_{n-1}} \nonumber \\
	& \geq \frac{q_n}{m+q_{n-1}} -  \frac{\lfloor q_n/2 \rfloor}{m+q_{n-1}} \nonumber\\
	& \geq \frac{q_n}{2(m+q_{n-1})} \nonumber \\
	& \geq \frac{q_n}{2(q_n(q_n+1)+q_{n-1})} \nonumber \\
	& \geq \frac{1}{3q_n} \nonumber \\
	& = \frac{L}{3}. \nonumber
\end{align}
Therefore this choice of $\kappa$, $L$, and $N_1$ ensures that Corollary \ref{sumtool} holds for the sequence $c_m$ defined by any $\alpha \in S[a_1\ldots a_n]$.
Let $A = 2B + 1$, where $B$ is the number of intervals in $U$.
For this $A>0$, Corollary \ref{sumtool} provides $N^*$. Choose an even $a_{n+1}$ so that $a_{n+1}q_n > N^*$.

For $\alpha \in S[a_1\ldots a_n a_{n+1}]$, we now return to $c_m'(\alpha)$, the un-adjusted sign pattern of the sum $\sum f\circ T^m(x)\frac{1}{m}$. By Lemma \ref{changes}, there are at most $2B$ values of $l$ at which $s(\sigma_l)\neq s(\sigma_{l+1})$. Let $l_1^* < l_2^* < \cdots < l_b^*$ be the values of $l$ where $s(\sigma_l)$ changes. Let 
\begin{eqnarray*}
	\tau_0 &=& [1, q_{n-1} + l_1^* q_n], \\
	\tau_i &=& [q_{n-1} + l_i^* q_n + 1, q_{n-1} + l_{i+1}q_n] 
		\hspace{3 em} \text{ for } 1 \leq i < b, \text{ and } \\
	\tau_b &=& [q_{n-1} + l_b^* q_n + 1, q_{n+1}].
\end{eqnarray*}
If for all $0 \leq i \leq b$, $|\sum_{m\in\tau_i}f\circ T^m(x)\frac{1}{m}|<1$, then we arrive at a contradiction to the fact, established above, that $\sum_{m=q_{n-1}+1}^{q_{n+1}}c_m \frac{1}{m}>A=2B+1$. So there must be at least one $\tau_i$ such that $|\sum_{m\in\tau_i}f\circ T^m(x)\frac{1}{m}|\geq1$. Taking $n_1$ and $n_2$ as the first and last integers in $\tau_i$, we note that $n_2 > n_1 > q_{n-1}$.

Since $S[a_1 \ldots a_n]$ is a nested sequence of closed subsets of the circle,
the intersection is nonempty. Take $\alpha \in \cap_n S[a_1 \ldots a_n]$. Let 
$q_{n}$ be the denominator of the $n^{th}$ convergent to $\alpha$. The continued
fraction expansion of $\alpha$ implies that $q_n \rightarrow \infty$ as $n \rightarrow \infty$. For each $n$, $\alpha \in S[a_1 \cdots a_{n+2}]$, so there exist $n_2>n_1>q_{n}$ such that $|\sum_{n=n_1}^{n_2} f\circ T^n(x)\frac{1}{n}|>1$. 
Hence, the series $\sum f\circ T^n(x)\frac{1}{n}$ diverges by the Cauchy criterion.
\end{proof}

As mentioned in the introduction, our argument actually proves a stronger result:

\begin{cor}
Let $\mathcal{F}_B$ be the set of all $f=2\chi_U -1$ where $m(U)=1/2$ and $U$ can be written as a finite union of no more than $B$ intervals. Let $b_n$ be any sequence such that $\sum n(b_{n+1}-b_n)$ diverges. Then there is a dense, uncountable set of irrational $\alpha$ such that $\sum f\circ T^n(x) b_n$ diverges for any $f\in\mathcal{F}_B$ and any $x$.
\end{cor}

\begin{proof}
First, our proof depends on $f$ only through the number of intervals used to write $U$, so we get the result for all $f\in\mathcal{F}_B$.

Second, the condition $\sum n(b_{n+1}-b_n)$ divergent is sufficient to run the argument of Lemma \ref{partslemma}, which drives divergence throughout the rest of the proof.

Finally, in the proof of Theorem \ref{divergent theorem}, one is free to chose the initial terms of the continued fraction expansion of $\alpha$, taking up the argument given there only after this initial segment. Note that in doing so, we must take up the argument at some $n^*$ where $q_{n^*}$ is odd, after which taking all $a_n$ even will ensure all subsequent $q_n$ are odd. It is easy to show from the recurrence relation for the $q_n$ and the fact that $q_0=1$ that there are infinitely many $n$ such that $q_n$ is odd, so there is no problem finding such an $n^*$. This allows us to find a dense set of irrational $\alpha$ which are divergent.

Our only requirement for divergence is that $a_n$ are sufficiently large and even. Since there are always infinitely many choices for each $a_n$, we can construct uncountably many divergent irrational $\alpha$.
\end{proof}


\section{Liouville numbers and convergence}\label{Liouvillesec}

\subsection{Theorem 2}

\begin{defn}
An irrational real number $\alpha$ is Liouville if for all $v\geq1$, there exists
a rational number $\frac{p}{q}$ such that 

\[ \big| \alpha - \frac{p}{q} \big| <q^{-(v+1)}. \]
\end{defn}

Note that the Liouville condition is equivalent to $\Vert q\alpha \Vert < q^{-v}$ where $\Vert x \Vert$ denotes the distance from the nearest integer. Let

\[ \mathcal{K}_v = \{ \alpha: \Vert q\alpha \Vert < q^{-v} \mbox{ infinitely often} \}. \]

\noindent By Dirichlet's theorem, $\mathcal{K}_1=\mathbb{R}$ and by a result of Khintchine, for $v>1$, $\mathcal{K}_v$ is a null set (see, e.g. \cite{cassels}). By a result of Jarn\'ik \cite{jarnik} and Besicovitch \cite{besicovitch}, for $v>1$, the Hausdorff dimension of $\mathcal{K}_v$ is $\frac{2}{v+1}$. This implies that the Hausdorff dimension of the set of Liouville numbers is zero.

Liouville numbers are very well approximated by rational numbers whose denominators are not too large. Similarly, the proof of Theorem \ref{divergent theorem} relies on constructing $\alpha$ which are very closely approximated by their convergents. So it is not too surprising that there is a connection between the two, and this is the content of Theorem 2.

As before, we assume that $f=2\chi_U -1$ where $U$ is a finite union of $B$ intervals on $S^1$ with $m(U)=\frac{1}{2}$. We use the sequence $b_n=\frac{1}{n}$.

\begin{thm2}\label{thm:NL implies C}
If $\alpha$ is divergent for $0$, a function $f=2\chi_U -1$, and $b_n=\frac{1}{n}$, then $\alpha$ is Liouville.
\end{thm2}

\noindent \emph{Idea of proof:} We prove the contrapositive. We assume that $\alpha$ 
is not Liouville and decompose the $\sum f\circ T^n(0)\frac{1}{n}$ into an infinite collection of (essentially) alternating subseries. Each such series is summable, and we can bound all its partial sums uniformly with ease. We then show that the sum of these bounds is finite, and that this implies that the original series is summable.

\subsection{A decomposition scheme}

First, we want a scheme for decomposing the sequence $(f\circ T^n(0)\frac{1}{n})$ into alternating sequences.
Throughout the following, we write $[a,b]$ for $\{a, a+1, \ldots, b\}$ and will refer to such subsets of the integers as \emph{intervals}. In the decomposition, we use nested intervals with lengths related to the denominators of the continued fraction expansion, $q_i$.

We write $(c_n)=(f\circ T^n(0))$ and $(\gamma_n) = (f\circ T^n(0)\frac{1}{n}).$ We will use roman letters ($d_n, b_n$) to denote subsequences of $(c_n)$ and greek letters ($\delta_n, \beta_n$) to denote the corresponding subsequences of $(\gamma_n)$.

We recall the Denjoy-Koksma Lemma:

\begin{lemma}\cite[VI Thm 3.1]{herman}
Let $f$ be any mean zero function on $S^1$. Let $[a,b]$ be any interval of length $q_k$. Then for any $x \in S^1$,

\[ \Big|\sum_{j\in[a,b]} f\circ T^j (x)\Big| < Var(f).\]
\end{lemma}

\begin{cor}\label{sumbound}
Let $f=2\chi_U-1$, where $U$ is the union of $B$ intervals and $m(U)=\frac{1}{2}$. Then, for any interval $[a,b]$ of length $q_k$ and any $x \in S^1$,

\[ \Big|\sum_{j\in[a,b]} f\circ T^j (x)\Big| < 4B.\]
\end{cor}

\begin{defn}
Sequence $(x_n)$ is a \emph{pair-permutation} of $(y_n)$ if for all $k\in\mathbb{N}$, $\{x_{2k-1},x_{2k}\}$ and $\{y_{2k-1},y_{2k}\}$ are equal as sets.
\end{defn}

In other words, $(y_n)$ is obtained from $(x_n)$ by permuting some pairs of adjacent terms.

\begin{defn}
We call a sequence \emph{near-alternating} if it is a pair-permutation of an alternating sequence.
\end{defn}

Let $Q_i = \prod_{j=1}^i q_j$. Then $Q_1 = q_1$ and $Q_{i+1} = Q_i q_{i+1}$ for all
$i \geq 1$. To avoid the use of floor functions and to remove indices as efficiently as possible, we recursively define a sequence of 0's and 1's. First, $\xi_1$ is 0 if $q_1 - 4B$ is even and 1 if odd. For $i > 1$, $\xi_i$ is 0 if $(4BQ_{i-2} + \xi_{i-1})q_i - 4BQ_{i-1}$ is even and 1 if odd. Then
\begin{equation*}
\left\lfloor \frac{(4BQ_{i-2} + \xi_{i-1})q_i - 4BQ_{i-1}}{2}\right\rfloor 
= \frac{(4BQ_{i-2} + \xi_{i-1})q_i - 4BQ_{i-1} - \xi_i}{2}.
\end{equation*}

Without loss of generality, assume that $q_1 > 4B$. If this is not the case,
the proof can be modified by shifting all of the indices.

\begin{prop}\label{decomp2}
There is a decomposition $(c_n)=\sqcup_{i=1}^\infty (d_n^{(i)})$ such that 
$(d_n^{(1)})$ is union of $(q_1 - 4B - \xi_1)/2$ subsequences, 
$(d_n^{(2)})$ is a union of $((4B + \xi_{1})q_2 - 4Bq_1 - \xi_2)/2$ subsequences, and $(d_n^{(i)})$ is the union of $((4BQ_{i-2} + \xi_{i-1})q_i - 4BQ_{i-1} - \xi_i)/2$ subsequences for $i > 2$. Each of these subsequences of some $(d_n^{(i)})$ is near-alternating.
\end{prop}

\begin{proof}

Throughout this proof, we will take ``a length $Q_i$ interval'' to mean an interval of the form $[(j-1)Q_i + 1, jQ_i]$ for some integer $j\geq 1$.

Let $(c_n^{(0)}) = (c_n) = (f\circ T^n(0))$. By Corollary \ref{sumbound}, each length $Q_1 = q_1$ interval contains at least $(q_1-4B-\xi_1)/2$ indices such that $c_n = +1$ and the same number of indices such that $c_n = -1$. For $l = 1, \ldots, (q_1-4B-\xi_1)/2$, let $b_{2j-1}^{(1,l)}$ be the $l^{\text{th}}$ term of $(c_n)$ that is equal to $+1$ with index in the $j^{\text{th}}$ length $Q_1$ interval. Similarly, let $b_{2j}^{(1,l)}$ be the $l^{\text{th}}$ term of $(c_n)$ that is equal to $-1$ with index in the $j^{\text{th}}$ length $Q_1$ interval. Let $(d_n^{(1)})$ be the union of these $(q_1-4B-\xi_1)/2$  near-alternating sequences. Let $X_1$ be the indices of the terms of $(c_n)$ that are not in $(d_n^{(1)})$. Then the intersection of each length $Q_1$ interval with $X_1$ has size $4B+\xi_1$.

Since each length $Q_2$ interval contains $q_2$ length $Q_1$ intervals, the intersection of $Q_2$ with $X_1$ has size $(4B+\xi_1)q_2$. Additionally, Corollary \ref{sumbound} implies that the sum over each length $q_2$ interval is at most $4B$. Since there are $q_1$ intervals of length $q_2$ in each length $Q_2$ interval, the sum over each length $Q_2$ interval is at most $4Bq_1$. Since $(d_n^{(1)})$ removes the same number of $+1$'s and $-1$'s from each length $Q_2 = q_1q_2$ interval, the sum over the remaining terms of each $Q_2$ interval (the terms also in $X_1$) is also at most $4Bq_1$. Thus, each length $Q_2$ interval contains at least $((4B + \xi_{1})q_2 - 4Bq_1 - \xi_2)/2$ remaining indices (in $X_1$) such that $c_n = +1$, and similarly for $-1$. As above, let $b_{2j-1}^{(2,l)}$ be the $l^{\text{th}}$ remaining $+1$ and $b_{2j}^{(2,l)}$ be the $l^{\text{th}}$ remaining $-1$, where the indices as terms of $(c_n)$ are in the $j^{\text{th}}$ length $Q_2$ interval and in $X_1$. Let $(d_n^{(2)})$ be the union of these near-alternating sequences. Let $X_2$ be the indices of the terms of $(c_n)$ that are not in $(d_n^{(1)})$ or in $(d_n^{(2)})$. Then the intersection of each length $Q_2$ interval with $X_2$ has size $4Bq_1+\xi_2$.

For the induction hypothesis, let $i \geq 2$ be an integer. Suppose that we have chosen disjoint sequences $(d_n^{(j)})$ for all $1 \leq j \leq i$. Let $X_i$ be the set of indices of the terms of $(c_n)$ that are not in any of the $(d_n^{(j)})$ so far. Suppose that the intersection of each length $Q_i$ interval with $X_i$ has size $4BQ_{i-1} + \xi_i$. 

Since each length $Q_{i+1}$ interval contains $q_{i+1}$ length $Q_i$ intervals, the intersection of $Q_{i+1}$ with $X_i$ has size $(4BQ_{i-1}+\xi_i)q_{i+1}$. Additionally, Corollary \ref{sumbound} implies that the sum over each length $q_{i+1}$ interval is at most $4B$. Since there are $Q_i$ intervals of length $q_{i+1}$ in each length $Q_{i+1}$ interval, the sum over each length $Q_{i+1}$ interval is at most $4BQ_{i}$. Since each $(d_n^{(j)})$ for $1 \leq j \leq i$ removes the same number of $+1$'s and $-1$'s from each length $Q_{i+1}$ interval, the sum over the remaining terms of each $Q_{i+1}$ interval (the terms also in $X_i$) is also at most $4BQ_i$. Thus, each length $Q_{i+1}$ interval contains at least $((4BQ_{i-1} + \xi_{i})q_{i+1} - 4BQ_{i} - \xi_{i+1})/2$ remaining indices (in $X_i$) such that $c_n = +1$, and similarly for $-1$. As above, let $b_{2j-1}^{(i+1,l)}$ be the $l^{\text{th}}$ remaining $+1$ and $b_{2j}^{(i+1,l)}$ be the $l^{\text{th}}$ remaining $-1$, where the indices as terms of $(c_n)$ are in the $j^{\text{th}}$ length $Q_{i+1}$ interval and in $X_i$. Let $(d_n^{(i+1)})$ be the union of these near-alternating sequences. Let $X_{i+1}$ be the indices of the terms of $(c_n)$ that are not in $(d_n^{(j)})$ for $1 \leq j \leq i+1$. Then the intersection of each length $Q_{i+1}$ interval with $X_{i+1}$ has size $4BQ_i+\xi_{i+1}$. Hence, the proof follows by induction.

\end{proof}

Let $(b_n^{(i,l)})$ be the near-alternating subsequences of $(d_n^{(i)})$ obtained in Proposition \ref{decomp2}; $n$ indexes within each individual sequence, and $l$ indexes the sequences themselves. Let $l$ index them so that $b_1^{(i,l)}$ always comes before $b_1^{(i,l+1)}$ as elements of $(c_n)$.

\begin{defn}
Let $ind(b_1^{(i,l)})$ denote the index of $b_1^{(i,l)}$ as an element of $(c_n)$.
\end{defn}

We have the following control on the first elements of the near-alternating sequences.

\begin{prop}\label{index}
For all $l$, and for $i>5$,
	\[ ind(b_1^{(i,l)}) \geq \Big\lfloor  \frac{l}{4BQ_{i-4}+\xi_{i-3}}  \Big\rfloor Q_{i-3}\]
and
	\[ ind(b_1^{(i,l)}) \geq \frac{(4BQ_{i-3}+\xi_{i-2})q_{i-1}-4BQ_{i-2}-\xi_{i-1}}{2}.   \]
\end{prop}

\begin{rmk}
The first bound is stronger for large $l$; the second bound's purpose is to give a nontrivial lower bound when $l<4BQ_{i-4}+\xi_{i-3}$.
\end{rmk}

\begin{proof}
For the first bound, we have the following argument. Examining the proof of Proposition \ref{decomp2}, we see that the terms of $(b_n^{(i,l)})$ have indices in $X_{i-3}$, i.e. among those indices which have not been used for $(d_n^{(j)})$ for $1\leq j \leq i-3$. As noted in that proof, the intersection of $X_{i-3}$ with each length $Q_{i-3}$ interval has size $4BQ_{i-4}+\xi_{i-3}$. Therefore, the $l^{th}$ index in $X_{i-3}$ is at least $\lfloor \frac{l}{4BQ_{i-4}+\xi_{i-3}}\rfloor Q_{i-3}$.

For the second bound, we note that in Proposition \ref{decomp2} the first 
\[ \frac{(4BQ_{i-3}+\xi_{i-2})q_{i-1}-4BQ_{i-2}-\xi_{i-1}}{2}\]
 $+1$'s and $-1$'s have been removed from $c_n$ at the $i-1^{st}$ step of the process, leaving behind $X_{i-1}$. The terms of $(b_n^{(i,l)})$ are drawn from $X_{i-1}$, hence the index of any remaining term has this lower bound. 
\end{proof}



Now that we have carefully extracted our near-alternating sequences from $(c_n)$ and carefully bounded the number of such sequences and the index of the first terms, we prove two lemmas on the growth rate of these quantities. These will simplify our convergence estimates in the next subsection.

\begin{lemma}\label{linearbound1}
Fix $\alpha$. There exists a constant $E$ (uniform in $i>2$) such that 
\[ \frac{(4BQ_{i-2} + \xi_{i-1})q_i - 4BQ_{i-1} - \xi_i}{2} \leq CQ_{i-2}q_i. \]
\end{lemma}

\begin{proof}
It is easy to check that 
\[ \frac{(4BQ_{i-2} + \xi_{i-1})q_i - 4BQ_{i-1} - \xi_i}{2} \frac{1}{Q_{i-2}q_i} \]
is uniformly bounded in $i$ since $Q_{i-2}q_i$ grows at least as fast as the first term.
\end{proof}

\begin{lemma}\label{linearbound2}
Fix $\alpha$. There exists a constant $F$ (uniform in $i>5$) such that 
\[ ind(b_1^{(i,l)}) \geq Fq_{i-3}l\]
for all $l$.
\end{lemma}

\begin{proof}
For $1\leq l \leq 4BQ_{i-4}+\xi_{i-3}$, using Proposition \ref{index},
\begin{align}
	ind(b_1^{(i,l)}) &\geq \frac{(4BQ_{i-3}+\xi_{i-2})q_{i-1}-4BQ_{i-2}-\xi_{i-1}}{2} \nonumber \\
				& \geq \frac{4BQ_{i-3} q_{i-1}-4BQ_{i-2}-\xi_{i-1}}{2} \nonumber \\
				& \geq BQ_{i-3} q_{i-1}-BQ_{i-2} \nonumber \\
				& = BQ_{i-3}(q_{i-1}-q_{i-2}) \nonumber \\
				& \geq BQ_{i-3} \nonumber \\				
				& \geq Bq_{i-3}(4BQ_{i-4}+\xi_{i-3})/(8B) \nonumber \\
				& \geq \frac{B}{8}q_{i-3}l \nonumber				
\end{align}
as desired.

For $l\geq 4BQ_{i-4}+\xi_{i-3}$, again by Proposition \ref{index}
\begin{align}
	ind(b_1^{(i,l)}) &\geq \Big\lfloor  \frac{l}{4BQ_{i-4}+\xi_{i-3}}  \Big\rfloor Q_{i-3} \nonumber \\
				& \geq \frac{l}{16B Q_{i-4}}Q_{i-3} \nonumber \\
				& = \frac{1}{16B}lq_{i-3}. \nonumber
\end{align}
Taking $F=\frac{1}{16B}$ finishes the proof.
\end{proof}

%

\subsection{The convergence argument}

We need the following pair of straightforward lemmas on sums involving near-alternating series and decompositions of series:

\begin{lemma}\label{lemma:alt}
Let $(\beta_n)$ be a decreasing sequence, with $|\beta_n|\to 0$, and such that exactly one of $\{ \beta_{2n-1},\beta_{2n}\}$ is positive for each integer $n$. Then $\sum_n \beta_n$ converges and 

\[ \Big|\sum_{n=1}^\infty \beta_n \Big| \leq |\beta_1|. \]
Furthermore, for any interval $[a,b]$, 

\[ \Big|\sum_{n\in[a,b]} \beta_n\Big| \leq 2|\beta_1|. \]
\end{lemma}

\begin{proof}
It is easy to verify that the sign pattern giving the largest value of the full sum is $(-1)^{n+1}$ and the pattern giving the smallest value is $(-1)^{n}$. The first statement is then a standard fact about alternating series.

The second statement follows from the first, after noting that it is possible that the $a$ and $a+1^{st}$ terms of $\beta_n$ may have the same sign. 
\end{proof}

\begin{rmk}
The proof actually gives 

\[  \Big|\sum_{n\in[a,b]} \beta_n\Big| \leq 2|\beta_a|\]
but we don't need this below.
\end{rmk}

\begin{lemma}\label{funnysumtool}
Suppose that we have a decomposition of a sequence $(\gamma_n) = \sqcup_i(\delta_n^{(i)})$ satisfying:
\begin{itemize}
	\item For all $i$, $\sum_n \delta_n^{(i)}$ converges.
	\item For all $i$ and all $[a,b]$, there exist $D^{(i)}$ such that $|\sum_{n\in[a,b]} \delta_n^{(i)}|\leq D^{(i)}$, and $\sum_i D^{(i)}<\infty$.
\end{itemize}
Then $\sum_n \gamma_n$ converges.
\end{lemma}

\begin{proof}
We prove convergence by the Cauchy Criterion. Let $\epsilon>0$ be given. For each $i \in \bN$, let $X_i$ be the set of indices of the terms from $(\gamma_n)$ that are in $(\delta_n^{(i)})$. Let $I\in\mathbb{N}$ be such that $\sum_{i > I} D^{(i)}<\epsilon/2$. For each $i\leq I$, let $N_i\in \mathbb{N}$ be such that for any $m_1,m_2\geq N_i$ the terms of $(\delta_n^{(i)})$ such that $\phi^{(i)}(n)\in[m_1,m_2]$ satisfy 

\[ \Big|\sum_{n:\phi^{(i)}(n)\in[m_1,m_2]} \delta_n^{(i)}\Big| = \Big|\sum_{[m_1,m_2]\cap X_i} \delta_n\Big| < \frac{\epsilon}{2^{i+1}}. \]

Let $N=\max\{N_i\}$. Then, for any $m_1,m_2\geq N$,
\begin{align}
	\Big| \sum_{n\in[m_1,m_2]} \gamma_n \Big| & = \Big| \sum_{i=1}^I \sum_{[m_1,m_2]\cap X_i} \delta_n + \sum_{i > I} \sum_{[m_1,m_2]\cap X_i} \delta_n \Big| \nonumber \\
	& \leq \sum_{i=1}^I \Big| \sum_{[m_1,m_2]\cap X_i} \delta_n \Big| + \sum_{i>I} \Big| \sum_{[m_1,m_2]\cap X_i} \delta_n \Big| \nonumber \\
	& \leq \frac{\epsilon}{2} + \sum_{i> I} D^{(i)} \nonumber \\
	&\leq \epsilon. \nonumber
\end{align}
\end{proof}

We recall a few facts relating the Liouville property to the continued fraction expansion.

Let $NL$ be the set of irrational non-Liouville numbers, and denote by $NL_v=\{ \alpha: \Vert q\alpha \Vert < q^{-v} \mbox{ has no solutions} \}$. Then $\cup_{v\geq 1} NL_v = NL$. Let us assume $\alpha \in NL_v$. Then we have that for all $q$, 

\[ \Vert q\alpha \Vert \geq q^{-v}. \]
It is a standard result in the theory of continued fractions (see, e.g. \cite[Thm 9]{khinchin}) that 

\[ \Vert q_k \alpha \Vert < \frac{1}{q_{k+1}} \mbox{ for all } k. \]
We thus have 

\[ \frac{1}{q_k^{v}} \leq \Vert q_k\alpha \Vert < \frac{1}{q_{k+1}} \]
and 

\[ \frac{q_{k+1}}{q_k^{v}}<1, \mbox{ or } q_{k+1} < q_k^{v} \mbox{ for all }k. \]

From the recurrence relation for $q_k$ and a comparison with the Fibonnacci sequence, we also have that there is some $\varphi>1, C>0$ such that

\[ q_k \geq  C \varphi^k.\]

We are now ready to complete our proof of Theorem \ref{thm2}.

\begin{proof}[Proof of Theorem \ref{thm2}]

We use the decomposition from Proposition \ref{decomp2} to decompose the sequence $(\gamma_n) = (f\circ T^n(0)/n)$ into subsequences $(\delta_n^{(i)})$. Using the decomposition $(c_n) = \sqcup_i(d_n^{(i)})$ and the associated index function, let $\delta_n^{(i)} = \gamma_{ind(d_n^{(i)})}$ and $\beta_n^{(i,l)} = \gamma_{ind(b_n^{(i,l)})}$. Then the decompositions $(\gamma_n) = \sqcup_i(\delta_n^{(i)})$ and $(\delta_n^{(i)})=\sqcup_l(\beta_n^{(i,l)})$ also satisfy Propositions \ref{decomp2} and \ref{index}.

To prove convergence using the mechanism of Lemma \ref{funnysumtool}, we need only obtain estimates for $i$ sufficiently large, so we restrict our attention to $i>5$. Then, using Lemma \ref{linearbound1} the sequence $(\delta_n^{(i)})$ consists of at most $CQ_{i-2}q_i$ near-alternating sequences $(\beta_n^{(i,l)})$. As before, they are indexed so that $\beta_1^{(i,l)}$ always comes before $\beta_1^{(i,l+1)}$ as elements of $(\gamma_n)$. The individual series $\sum_n \beta_n^{(i,l)}$ converge by Lemma \ref{lemma:alt}, with 
\[ |\sum_{n\in [a,b]} \beta_n^{(i,l)}| \leq 2|\beta_1^{(i,l)}|.\]
Using Lemma \ref{linearbound2}
\[ |\beta_1^{(i,l)}| \leq \frac{1}{Fq_{i-3}l}.\]

Applying these bounds and Lemma \ref{lemma:alt}, we get a bound on partial sums of $\delta_n^{(i)}$ as follows: for any interval $[a,b]$,

\begin{align}
	|\sum_{n\in[a,b]} \delta_n^{(i)} | &\leq 2\sum_{l=1}^{EQ_{i-2}q_i} |\beta_1^{(i,l)}| \nonumber \\
						& \leq 2\sum_{l=1}^{EQ_{i-2}q_i} \frac{1}{Fq_{i-3}l}. \nonumber 
\end{align}
This can be bounded above by

\begin{align} 
	\frac{2}{Fq_{i-3}}( 1+ \log (E Q_{i-2}q_i)) & \leq \frac{2}{Fq_{i-3}}(1+ \log (Eq_i^{i-1})) \nonumber \\
	& = \frac{2}{Fq_{i-3}}( 1+ \log E +(i-1) \log q_i) \nonumber \\
	& \leq \frac{2}{Fq_{i-3}}( 1+ \log E + (i-1) \log q_{i-3}^{v^3}), \nonumber
\end{align}
using that $q_i<q_{i-3}^{v^3}$ by the Liouville condition.

By the exponential growth of $q_k$, $\sum_k \frac{1}{q_k}<\infty$. We claim that $\sum_k \frac{k \log q_k}{q_{k}}$ converges, which will ensure that $\sum_i D^{(i)}$ converges and permit application of Lemma \ref{funnysumtool}. Note that for $x>e$, $\frac{\log x}{x}$ is decreasing and recall that $q_k \geq C \varphi^k$. Since $q_3$ is surely greater than $e$, for $k\geq 3$, $\frac{k \log q_k}{q_k} \leq \frac{k \log C\varphi^k}{C\varphi^k}$ and

\begin{align}
	\sum_{k=1}^\infty \frac{k \log q_k}{q_k} & \leq \frac{ \log q_1}{q_1}+ \frac{2 \log q_2}{q_2} + \sum_{k=3}^\infty \frac{k \log C\varphi^k}{C\varphi^k} \nonumber \\
	& = \frac{\log q_1}{q_1} + \frac{2\log q_2}{q_2} + \sum_{k=3}^\infty \frac{k\log C}{C\varphi^k}+\frac{k^2\log\varphi}{\varphi^k}\nonumber.
\end{align}
As $\varphi>1$, this converges.

We then have that $\sum_{i=1}^\infty \Big| D^{(i)} \Big|$ converges. By Lemma \ref{funnysumtool}, this implies that $\Big| \sum_{n=1}^\infty \gamma_n \Big|$ converges, establishing the theorem.

\end{proof}

As a corollary, we obtain the following result on the size of the set of divergent $\alpha$.

\begin{cor}
The set of all divergent $\alpha$ has Hausdorff dimension 0.
\end{cor}

%

\section{Convergent Liouville numbers}\label{conv Liou}

We are left with the question of whether all Liouville $\alpha$ are divergent. The proof we noted in the introduction about convergence for non-Liouville $\alpha$ does not touch this question. We remark that the argument for divergence in Section \ref{sec:divergent} uses the odd parity of the $q_n$'s for the $\alpha$ we constructed. This certainly does not hold for all Liouville $\alpha$. We do not know if some sort of requirement on parity of the $q_n$'s is necessary to ensure divergence, but we will show here that, independent of this concern, there are convergent Liouville numbers. Specifically we will construct Liouville numbers for which the arguments of the previous section can still be used to prove convergence.

\begin{thm3}
There exist Liouville numbers $\alpha$ which are convergent for any $f=2\chi_U-1$, where $U$ is a union of finitely many intervals with $m(U)=\frac{1}{2}$, and any $x\in S^1$. The set of such $\alpha$ is dense.
\end{thm3}

\begin{proof}
The convergence argument of the previous section is our main tool.

It is standard that

\begin{equation}\label{upperbound}
	 \Vert q_k\alpha \Vert < \frac{1}{q_{k+1}}.
\end{equation}

Now suppose that we define $\alpha$ by choosing $a_{k+1}= q_k^{k-1}$. Note that as $q_k$ is defined only in terms of $a_1, \ldots, a_k$, this defines $\alpha$ inductively. We can take up this inductive definition after any initial sequence $[a_1a_2\ldots a_n]$, producing a dense set of $\alpha$. Then,

\[ q_{k+1}=a_{k+1}q_k+q_{k-1} \geq q_k^k. \]
With \eqref{upperbound}, this implies that for all $k$,

\[ \Vert q_k\alpha \Vert < \frac{1}{q_k^k}. \]
Thus, for any $v\geq 1$, all $\frac{p_k}{q_k}$ with $k\geq v$ satisfy the approximation condition in the definition of a Liouville number, so $\alpha$ is Liouville.

On the other hand, 

\begin{equation}\label{qkbound}
	 q_{k+1} \leq 2q_k^{k}
\end{equation}
at least for $k>1$. 

Examining the proof of Theorem \ref{thm:NL implies C}, we see that the key necessity for convergence of the series is summability of $\frac{k \log q_k}{q_{k-1}}$. By equation \eqref{qkbound},

\[\frac{k \log q_k}{q_{k-1}} \leq \frac{k \log 2q_{k-1}^{k-1}}{q_{k-1}} = \frac{k\log 2+ k(k-1)\log q_{k-1}}{q_{k-1}}. \]
Clearly $\sum_k \frac{k \log 2}{q_{k-1}}$ converges. For the second summand, again using $q_k\geq C\varphi^k$, 

\begin{align}
	 \sum_{k=1}^\infty \frac{k(k-1) \log q_{k-1}}{q_{k-1}} & \leq \frac{2 \log q_1}{q_1}+ \frac{6 \log q_2}{q_2} + \sum_{k=3}^\infty \frac{k(k-1)\log C\varphi^k}{C\varphi^k} \nonumber \\
	 & = \frac{2 \log q_1}{q_1}+ \frac{6 \log q_2}{q_2} + \sum_{k=3}^\infty \frac{k(k-1)\log C}{C\varphi^k}+\frac{k^2(k-1)\log\varphi}{C\varphi^k}. \nonumber
\end{align}
The sum converges, so by the argument of the proof for Theorem \ref{thm:NL implies C}, $\alpha$ is convergent for any $f$ as above and any $x$.

\end{proof}

Of course, the proof of Theorem \ref{thm2} presented here allows many other constructions of convergent Liouville numbers. Examining the proof above, we see that a sufficient condition for convergence is having $\frac{k\log q_k}{q_{k-3}}$ summable. This leaves plenty of leeway to choose $a_k$ sufficiently large to produce a Liouville number.

\appendix

\section{Convergence for non-Liouville numbers}\label{appendix}

We provide a second proof of Theorem \ref{thm2}, working out Kakutani and Petersen's remark from \cite{pk} that convergence follows from estimates of the discrepancy for non-Liouville numbers.

Let $\omega = ( x_1, x_2, \ldots)$ be a sequence of elements in $[0,1]$. For our work we will take $\omega = (n\alpha+x \ (mod \ 1))_n$. The failure of equidistribution of this sequence is measured by the discrepancy function:

\begin{defn}\label{discrep}
The \emph{discrepancy} of $\omega$ is
\[ D_N = D_N(\omega) = \sup_{0\leq \alpha < \beta \leq 1} \Big| \frac{\#([\alpha,\beta)\cap \omega|_{[1,N]})}{N} - (\beta-\alpha) \Big|.\]
\end{defn}

By the \emph{discrepancy of $\alpha$}, or $D_N(\alpha)$ we will mean $D_N((n\alpha + x \ (mod \ 1)))$. It is immediate from Definition \ref{discrep} that this function of $N$ is independent of $x$.

We recall the following definition and its connection to the Liouville property.

\begin{defn}
Let $\eta >0$. We say \emph{$\alpha$ is of type $\eta$} if $\eta = \sup \gamma$ such that 
\[ \liminf_{q\to\infty} q^\gamma \langle\langle q\alpha \rangle\rangle =0 \quad \mbox{ where } q\in\mathbb{N}.\]
As before $\langle\langle - \rangle\rangle$ denotes distance from the nearest integer.
\end{defn}

It is easy to check that $\alpha$ is Liouville if and only if $\alpha$ is of type $\infty$. By Dirichlet's approximation theorem, all numbers are of type at least 1.

The following result on the discrepancy of non-Liouville $\alpha$ is the key tool we need:

\begin{thm}\label{alpha discrep}
\emph{(See, e.g. \cite[Thm 3.2]{kn})}
Let $\alpha$ be of type $\eta$. Then for all $\epsilon>0$, $D_N(\alpha) = O(N^{-1/\eta+\epsilon}).$
\end{thm}

We now prove Theorem \ref{thm2} using this result.

\begin{proof}[Proof of Theorem \ref{thm2}]
Let $\alpha$ be non-Liouville; suppose it is of type $1\leq\eta<\infty$. Fix any $x\in[0,1)$. Let $S_nf = \sum_{i=1}^n f(n\alpha+x).$ Using summation by parts,
\[ \sum_{n=1}^N \frac{f(n\alpha+x)}{n} = \frac{S_Nf}{N} + \sum_{n=1}^{N-1} \frac{S_nf}{n(n+1)}.  \]
Using Theorem \ref{alpha discrep} and picking $\epsilon>0$ so small that $-1/\eta+\epsilon<0$, one can easily show that $|S_nf| = O(n^{1-1/\eta+\epsilon})=O(n^{1-t})$ for some $t>0$.  From this we immediately have that $\frac{S_Nf}{N}\to 0$ as $N\to\infty$ and that $\sum_{n=1}^\infty \frac{S_nf}{n(n+1)}$ converges, completing the proof.
\end{proof}

\bibliographystyle{alpha}
\bibliography{bib}

\end{document}